\pdfoutput=1
\documentclass[11pt]{mathdoc}

\usepackage{mcnotation}

\usepackage{tikz}
\usepackage[shortlabels, inline]{enumitem}
\usepackage{amsfonts,mathtools,dsfont}
\usepackage{xcolor}
\usepackage{IEEEtrantools}
\usepackage{graphicx}
\usepackage{subcaption}
\usepackage{accents}

\usepackage[capitalize,noabbrev]{cleveref}
\crefname{assumption}{Assumption}{Assumptions}
\crefformat{equation}{Equation~#2#1#3}
\crefformat{assumption}{Assumption~#2#1#3}
\crefformat{figure}{Figure~#2#1#3}

\newtheorem{lemma}{Lemma}

\newtheorem{assumption}{Assumption}

\title{Predictive performance of power posteriors%
\blankfootnote{Code to replicate the results is freely available at \url{https://github.com/yannmclatchie/power-posterior-prediction}.}%
}
\author[1]{Yann McLatchie}
\author[2]{Edwin Fong}
\author[3]{David T. Frazier}
\author[1]{Jeremias Knoblauch}
\affil[1]{Deparment of Statistical Science, Univesity College London}
\affil[2]{Department of Statistics and Actuarial Science, The University of Hong Kong}
\affil[3]{Department of Econometrics and Business Statistics, Monash University}
\makeatletter
\def\runauthor{McLatchie, Fong, Frazier, and Knoblauch}
\let\inserttitle\@title
\def\runtitle{\inserttitle}
\makeatother

\begin{document}
\maketitle
\begin{abstract}
    We analyse the impact of using tempered likelihoods in the production of posterior predictions. 
While the choice of temperature has an impact on predictive performance in small samples, we formally show that in moderate-to-large samples, tempering does not impact posterior predictions.

\end{abstract}
\section{Introduction}
Bayesian inference has become a popular framework for inference due to the incorporation of prior beliefs and automatic uncertainty quantification.
Traditional Bayesian analysis, however, is only optimal under the strong assumption that the model is well-specified \citep{bernardo_bayesian_1994,aitchison_goodness_1975}.
When this assumption is violated and the true data-generating distribution is outside the model class, the standard Bayes posterior can deliver unreliable inference \citep{owhadi_brittleness_2015,bissiri_general_2016,jewson_principles_2018, knoblauch_optimization_2022}.
A popular remedy are \textit{power posteriors} \citep{holmes_assigning_2017,grunwald_inconsistency_2017}, also known as \textit{fractional} \citep{bhattacharya_bayesian_2019} or $\alpha$\textit{-posteriors} \citep{yang_alpha_2020}.
These temper the likelihood $f_\theta$ by raising it to a power $\tau\in\mathbb{R}^+$, a constant referred to as the \textit{temperature} or \textit{learning rate}.
Denoting the prior density over $\theta$ by $\pi$, and $\y$ as the observed data, the power posterior is
\begin{equation}\label{eq:power-posterior}
    \pi_n^{(\tau)}(\theta\mid\y) \propto \pi(\theta)f_\theta(\y)^\tau.
\end{equation}
For $\tau=1$, this recovers the Bayes posterior, which optimally processes information if the model is correctly specified \citep{zellner_optimal_1988}, while $\tau < 1$ can improve robustness to model misspecification in small-to-moderate sample sizes \citep{grunwald_inconsistency_2017}.
Further, \citet[Equation 3.5]{miller_robust_2019} argue that \Cref{eq:power-posterior} is an approximation of Bayes' Theorem when conditioning on neighbourhoods of the observed data and show asymptotically that taking $\tau<1$ can provide inferential benefits.
Similarly, \citet{bhattacharya_bayesian_2019} use asymptotic arguments to argue that power posteriors have preferable concentration and generalisation properties.

To find an appropriate value for $\tau$, various contributions have focused on frequentist calibration \citep[see e.g.][]{syring_calibrating_2019,lyddon2019general,altamirano_robust_2023,matsubara_generalized_2023}, expected information matching \citep{holmes_assigning_2017}, and the so-called \textit{SafeBayes} approach \citep{grunwald2012safe, grunwald_inconsistency_2017}.
For a comparison of such methods, see \citet{wu_comparison_2023}.
Choosing $\tau$ for optimal predictive performance, however, is under-studied.
This is surprising given the growing interest in Bayesian prediction \citep{fortini_predictive_2012,kenett_predictive_2016,fortini_exchangeability_2024,fong_martingale_2023}.
In the remainder, we study the effect of $\tau$ on predictive performance in settings where the sample size is large enough for posterior concentration to be in evidence.
The results are intriguing: while varying $\tau$ may yield predictive benefit in small samples, in moderate-to-large samples the choice of $\tau$ has no meaningful impact on predictive performance. Indeed, trying to choose the value of $\tau$ that provides optimal predictive performance leads to an ill-defined optimisation problem.
\section{A predictive view on power posteriors}
\subsection{Choosing the temperature predictively}\label{sec:predictive-temp}
With the rise of \textit{algorithmic modelling} \citep{breiman_statistical_2001}, mainstream research has increasingly emphasised the prediction of observables.
For a Bayesian statistician, this amounts to focusing on the \textit{posterior predictive}, which integrates out parameter uncertainty via
\begin{equation*}
    p_n^{(\tau)}(\cdot \mid \y) = \int f_\theta(\cdot\mid\y) \dt \pi_n^{(\tau)}(\theta \mid \y).
\end{equation*}
For large but finite values of $\tau$, the prior is discounted and the posterior $\pi_n^{(\tau)}(\theta\mid\y)$ places most of its mass in a small neighbourhood around the point estimator $\hat\theta_n=\argmax_{\theta} \log f_\theta(\y)$, so that the resulting posterior predictive resembles the \textit{plug-in predictive} $f_{\hat \theta_n}(\cdot\mid\y)$.
Conversely, for small values of $\tau$, the posterior predictive resembles the \textit{prior predictive} $\int f_\theta(\cdot\mid\y){\dt\pi(\theta)}$.
A na\"ive approach would be to assume that there exists a unique choice for $\tau$ that interpolates the two in a predictively optimal manner, and in turn optimise for it.
For instance, in the idealised case where one knows the true predictive distribution $\preddens_n(\cdot\mid\y)$, one could attempt to choose $\tau$ so that its distance from the posterior predictive $p_n^{(\tau)}(\cdot\mid\y)$ is minimised.
In this paper we show that this is not possible:
both theoretically and empirically, once $n$ is even just moderately large, all values of $\tau$ taken over any positive and compact interval deliver a $p_n^{(\tau)}(\cdot\mid\y)$ with virtually identical predictive performance.
Our findings are two-fold:
first, we show that there is little theoretical benefit to choosing $\tau$ predictively.
Second, if, regardless, we choose tau by optimising predictive performance then we demonstrate that this results in an ill-defined optimisation problem.
Before formally stating this result, we demonstrate this behaviour on a simple numerical example.
\subsection{Normal location example}\label{sec:normal-location-example}
\begin{figure}[t!]
    \centering
    \includegraphics{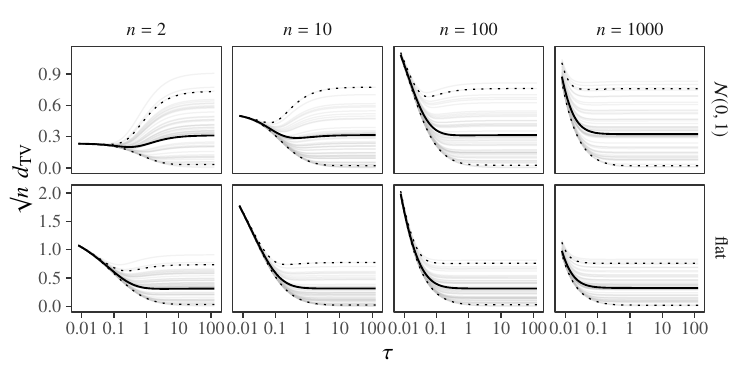}
    \vspace*{-1em}
    \caption{$\surd{n}$-scaled total variation between the power posterior predictive $p_n^{(\tau)}(\cdot \mid \y)$ of  a normal location model and the true predictive $\preddens_n(\cdot\mid\y)$.
    Grey curves correspond to individual dataset replicates,  dotted black lines to $5\%$ and $95\%$ quantiles, and solid black curves to expectation.}
    \label{fig:normal-location-tvd}
\end{figure}
We simulate $n$ independent and identically distributed observations from a Gaussian distribution with zero mean and unit variance, construct posterior predictives $p_n^{(\tau)}$ based on the normal likelihood, and then {compare this against the infeasible true predictive, denoted as $\preddens_n(\cdot\mid\y)$, in total variation distance,  by computing  $d_{\mathrm{TV}}\{\preddens_n(\cdot\mid\y), p_n^{(\tau)}(\cdot \mid \y)\}$} over $1,000$ data replications while varying $0.01 \leq \tau \leq 100$.
We show the average total variation distance, across the replications, and several individual replicates in \Cref{fig:normal-location-tvd}.
{Since the model is correctly specified, we can show that} $d_{\mathrm{TV}}\{\preddens_n(\cdot\mid\y), {f_{\hat \theta_n}(\cdot\mid\y)}\}$ vanishes  at rate $\surd{n}$ in this example, which is why we scale it with $\surd{n}$ to aid visualisation.
We repeat this for both a weakly-informative and a flat prior.
The resulting plot exposes an intriguing phenomenon:
for $\tau$ {on a compact interval}, the distance is essentially flat, so that identifying a predictively optimal value for $\tau$ is numerically fragile.

This suggests that the posterior predictive distribution is indistinguishable from the plug-in predictive once $\tau$ exceeds a critical threshold that appears to scale as $n^{-1/2}$.
Of course, in small samples, tempering can play an important role predictively \citep{grunwald_inconsistency_2017};
see~\Cref{sec:normal-location-supp} in the supplementary material and~\Cref{sec:additional-numerical-experiments} below for further discussion. {Importantly, none of these phenomena depend on correct model specification: the same behaviour is observed when the model is misspecified.
Indeed, our  theoretical findings identify the key condition for these results to be  posterior concentration.}
\section{The temperature is eventually inconsequential to predictive accuracy}\label{sec:theory}
\subsection{Technical results}
\label{sec:technical-results}
In this section, we show that as $n$ increases, $p_n^{(\tau)}(\cdot\mid\y)$ and ${f_{\hat \theta_n}(\cdot\mid\y)}$ become uniformly close over $\tau$.
Hence, in large samples, varying $\tau$ will not improve predictive performance.

For simplicity, assume that $\tau$ lies in a positive {and compact} interval, {$\tau \in [\underline{\tau}, \overline{\tau}]$ for constants $0<\underline{\tau} < \overline{\tau} < \infty$}.
{To state our assumptions, we take $d: \Theta \times \Theta \to [0,\infty)$ to be a divergence so that $d(\theta,\theta')\ge0$ for all $\theta,\theta'\in\Theta$ and $d(\theta,\theta
)=0$ implies that $\theta=\theta'$.
This divergence can be taken to be the Euclidean distance when $\Theta\subseteq\mathbb{R}^{d}$, but can take other forms for more general parameter spaces.} %
Further, define $\mathsf{L}(\theta)=\lim_{n\to\infty} n^{-1}\log f_\theta(\y)$, $\theta^{\star}=\argmax_{\theta} \mathsf{L}(\theta)$ as the population-optimal value for $\theta$, and $\dgp$ as the distribution from which the observations $\y$ are drawn.
{Throughout, we do not assume that the model is well-specified, so that there need not exist a $\theta^\star\in\Theta$ for which $f_{\theta^\star} = \preddens_n$.
Instead, we posit two generic assumptions:
the first imposes a mild technical condition  satisfied for a wide range of models, including all regular parametric statistical models, as well as a range of non-parametric models.
The second is a posterior concentration condition, and is likewise satisfied in a very large class of settings.}
\begin{assumption}
\label{ass:lipz}
For $\lambda$ denoting the Lebesgue measure, for some $\varepsilon>0$, any $\theta$ such that $d(\theta,\theta^{\star})\le\varepsilon$, and any $\y$, there exists a constant $0<M_\varepsilon<\infty$, independent of $\theta$ and $\y$, so that
\begin{IEEEeqnarray}{rCl}
    \label{eq:lipz}
    \int \left\{f_\theta(x\mid\y)^{1/2}-f_{\theta^{\star}}(x\mid\y)^{1/2} \right\}^2\dt\lambda(x)\le M_\varepsilon d(\theta,\theta^{\star})^2.
\end{IEEEeqnarray}
\end{assumption}
This assumption is similar to differentiability in quadratic mean, which is satisfied by statistical models with positive Fisher information at $\theta^\star$ \citep[][Lemma 7.6]{vaart_asymptotic_1998}.
While the assumption is given in its most general form, the dependence of $M_{\varepsilon}$ on $\varepsilon$ can be suppressed entirely if a sufficiently large constant can be found that satisfies the condition uniformly over shrinking values of $\varepsilon$: a situation common in both parametric and non-parametric settings.
\begin{assumption}\label{ass:concentration}
Take $\varepsilon>0$, $K>0$, $C>0$, and $A_\varepsilon=\{\theta\in\Theta:d(\theta,\theta^{\star})\le K\varepsilon\}$.
There exists a sequence $\varepsilon_n>0$, $\varepsilon_n\to0$, {$n\varepsilon_n^2/M_{\varepsilon_n}\rightarrow\infty$}, such that, for $K$ sufficiently large,
\begin{equation*}
\int \mathds{1}_{\left\{\theta\in A_{\varepsilon_n}^c\right\}}\, {\dt\pi_n^{(\tau)}(\theta\mid \y)} \le \exp(-Cn\tau\varepsilon_n^2K^2)
\end{equation*}
with $\mathbb{P}$-probability at least $1-\exp(-Cn\tau\varepsilon_n^2K^2)$.
Further, there is a sequence $\nu_n\rightarrow0,$ as $n\rightarrow\infty$, so that $d(\hat\theta_n,\theta^\star)\le \nu_{n}/M^{1/2}_{\nu_n}$,  with $\mathbb{P}$-probability at least $1-\nu_n$, with $M_{\nu_n}$ as in \Cref{ass:lipz} and such that {$M_{\nu_n}\ll\nu_n$}.
\end{assumption}
\Cref{ass:concentration} says that as $n$ increases, and in high probability, $\pi_n^{(\tau)}(\theta\mid \y)$ allocates an increasing amount of its probability mass onto a ball containing $\theta^\star$, and that $\hat\theta_n$ approaches $\theta^\star$ at rate $\nu_n$.
Rather than directly assume posterior concentration as we have done, one may instead derive this assumption, and subsequent results, from sufficient conditions similar to those of \citet[Theorems 1 and 2]{martin_direct_2022}.
In \Cref{sec:additional-results} of the supplementary material, we demonstrate that \Cref{ass:concentration} is satisfied under some well-understood regularity conditions.
While we will later empirically investigate only finite-dimensional models, our maintained assumptions do not limit us to this setting.
Indeed, our results apply to models of growing dimension and non-parametric models.
\begin{lemma}\label{lemma:tvd}
Under \cref{ass:lipz,ass:concentration}, for any $0<\underline{\tau}< \overline{\tau} < \infty$ and $\tau\in[\underline{\tau}, \overline{\tau}]$,
\begin{equation}\label{eq:tvd-bound}
    \tvdcurl{{f_{\hat \theta_n}(\cdot\mid\y)}}{p_n^{(\tau)}(\cdot\mid\y)} \le 2\max\left\{\varepsilon_n+\exp(-Cn\tau\varepsilon_n^2/ M_{\varepsilon_n}), \nu_n\right\}
\end{equation}
with $\mathbb{P}$-probability at least $1-2\max\left\{\varepsilon_n+\exp(-Cn\tau\varepsilon_n^2/M_{\varepsilon_n}), \nu_n\right\}$.
\end{lemma}
\Cref{lemma:tvd} shows that the difference between $p_n^{(\tau)}(\cdot\mid\y)$ and the plug-in predictive ${f_{\hat \theta_n}(\cdot\mid\y)}$ vanishes uniformly {over $\tau\in[\underline\tau,\overline\tau]$}, where the rapidity of this convergence depends on the rate of the plug-in estimator, $\nu_n$, and the rate of posterior concentration, $\varepsilon_n$.

In \Cref{sec:additional-results} of the supplementary material,
 we derive a similar result in expectation over $\y$ using the same conditions, (see \Cref{lemma:tvd-expectation}).
Furthermore, we also demonstrate that these results  extend to the case where $\tau = \tau_n$ with $\tau_n\rightarrow0$ as $n\rightarrow\infty$, and $n\tau_n\rightarrow\infty$ (see Lemmas \ref{lem:extra1} and \ref{lem:extra_conv}).
As it is not empirically relevant, the limiting case $\tau\rightarrow\infty$ is not considered.
\subsection{Interpretation}
\label{sec:interpretation}
\Cref{lemma:tvd} shows that the choice of $\tau$ has vanishing impact on prediction as $n\rightarrow\infty$.
As a result, attempting to optimise $\tau$ for predictive performance produces a range of values of $\tau$ that have indistinguishable predictive accuracy.
This finding is robust: it holds in high probability (\Cref{lemma:tvd}), in expectation (\Cref{lemma:tvd-expectation} in the supplement), for sequences $\tau_n\to0$ as $n\to\infty$ (\Cref{sec:additional-results} in the supplement),  and for distances other than total variation (see Section~\ref{sec:loo-cv}).

\Cref{lemma:tvd} provides a rigorous explanation for the behaviour in \Cref{fig:normal-location-tvd}:
as the posterior concentrates, $p_n^{(\tau)}(\cdot\mid\y)$ with $\tau\in[\underline{\tau}, \overline{\tau}]$, becomes indistinguishable from the oracle predictive $f_{\theta^\star}(\cdot\mid\y)$, which is itself indistinguishable from the plug-in predictive ${f_{\hat \theta_n}(\cdot\mid\y)}$, neither of which depend on $\tau$.
As the normal location model in \Cref{sec:normal-location-example} suggests, this behaviour occurs even for  moderate sample sizes.
Of course, our results are asymptotic in nature, and thus not necessarily indicative of the behaviours observed in small sample sizes.
Indeed, for small $n$, \cite{grunwald_inconsistency_2017} demonstrate that tempering can have predictive benefits.
\subsection{Applicability to generalised Bayes}\label{sec:gen-bayes}
Power posteriors are a special case of generalised Bayes posteriors \citep{bissiri_general_2016, knoblauch_optimization_2022}.
Indeed for any loss $\Ln(\theta, \y)$ whose parameters $\theta$ index a statistical model $f_\theta$ satisfying \Cref{ass:lipz}, \Cref{lemma:tvd} applies equally to the predictive $p_n^{(\tau)}(\cdot\mid\y,\Ln) = \int f_\theta(\cdot\mid\y){\dt\pi_n^{(\tau)}(\theta\mid\y,\Ln)}$ constructed from a Gibbs measure $\pi_n^{(\tau)}(\theta\mid\y,\Ln) \propto \pi(\theta) \exp\{-\mathsf{L}(\theta, y_{1:n})\}$ that satisfies \Cref{ass:concentration}.
In settings like this, the choice of $\tau$ still matters greatly, as it calibrates the weight of a data-dependent loss relative to the prior \citep[see e.g.][]{knoblauch_doubly_2018,altamirano_robust_2024,matsubara_robust_2022}.
Consequently, predictive performance for generalised Bayes posteriors is also largely independent of $\tau$, provided the sample size is sufficiently large.
While this was suggested by the experiments in \citet{loaiza2021focused} and \citet{frazier_loss-based_2025}, \Cref{lemma:tvd} provides the first rigorous and general proof of this fact.
\section{Cross-validation and the Kullback-Leibler divergence}
\label{sec:loo-cv}
Total variationgenerally not used to select hyper-parameters like $\tau$.
Instead, one would typically resort to leave-one-out cross-validation of the expected log predictive density induced by $\tau$, denoted $\elpdBiomet(\tau)$, and define
\begin{equation}\label{eq:tau-loss-cv}
    \tau^\star_{\mathrm{CV}} = \argmax_{\tau\in\mathbb{R}^+} \elpdBiomet(\tau).
\end{equation}
For the specific definition of $\elpdBiomet(\tau)$ and additional details,
see \Cref{sec:normal-location-supp} in the supplement.

\begin{figure}[t!]
    \centering
    \includegraphics[width=\textwidth]{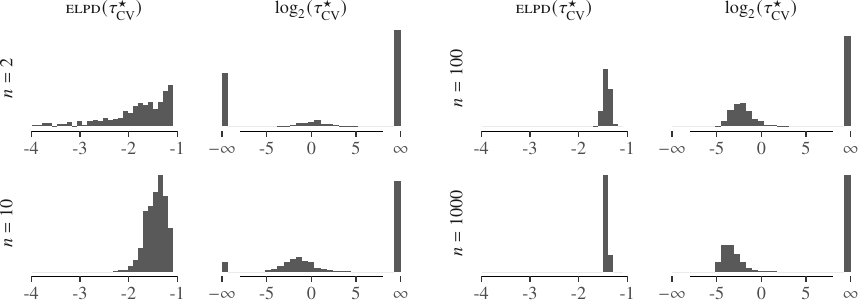}
    \vspace*{-1em}
    \caption{Histograms of $\elpdBiomet(\tau_{\mathrm{CV}}^\star)$ and $\tau_{\mathrm{CV}}^\star$ in a
    normal location model with standard normal prior.}
    \label{fig:normal-location-tau-selection}
\end{figure}
We examine this proposal for temperature selection in the normal location example of \Cref{sec:normal-location-example}.
In \Cref{fig:normal-location-tau-selection}, we plot the distributions of $\elpdBiomet(\tau_{\mathrm{CV}}^\star)$ and $\log_2(\tau_{\mathrm{CV}}^\star)$ over $1,000$ data replicates.
When $n$ is small, the distribution of $\tau^\star_{\mathrm{CV}}$ places most of its mass either on the prior predictive ($\tau_{\mathrm{CV}}^\star = 0$) or the plug-in predictive ($\tau_{\mathrm{CV}}^\star = \infty$).
As $n$ grows, the distribution increasingly shifts its mass away from the prior predictive.
Observe that if there is \textit{any} non-zero probability of selecting $\tau^\star_{\mathrm{CV}} = \infty$, the variance of the estimator defined in~\Cref{eq:tau-loss-cv} is infinite.
In the supplement, we perform additional experiments using cross-validation to predictively choose $\tau$, and further confirm that the conclusions are the same as for total variation case.

Since $\elpdBiomet(\tau)$ approximates the Kullback-Leibler divergence $d_{\mathrm{KL}}\{\preddens_n(\cdot\mid\y); p_{n}^{(\tau)}(\cdot\mid \y, \Ln)\}$, we prove a similar result to \Cref{lemma:tvd} for $d_{\mathrm{KL}}$ and show that as $n\rightarrow\infty$, the choice of $\tau$ does notimpact predictive performance as measured by $d_{\mathrm{KL}}$.
\begin{lemma}\label{lemma:kl}
Under \cref{ass:concentration}, for any $0<\underline{\tau}< \overline{\tau} < \infty$ and $\tau\in[\underline{\tau}, \overline{\tau}]$,
\begin{equation*}
\kldcurl{f_{\theta^\star}(\cdot\mid\y)}{p_n^{(\tau)}(\cdot\mid\y,\Ln)} \le\varepsilon_n^2+\exp(-Cn\tau\varepsilon_n^2)+o(1),
\end{equation*}
with $\mathbb{P}$-probability at least $1  - \exp(-Cn\tau\varepsilon_n^2)$.
\end{lemma}
\section{Additional numerical experiments}\label{sec:additional-numerical-experiments}
We further explore our findings in two additional examples.
First, we show that  contrary to the conclusions one might draw from \Cref{fig:normal-location-tvd}, the plug-in predictive is unsafe, and tempering can improve predictions in small samples.
Second, we demonstrate that our results remain valid under model misspecification.

Casual reading of our results may suggest the plug-in predictive ($\tau = \infty$) as a sensible default choice for prediction.
This is not the case: our results only bound the difference between the posterior predictive and the plug-in predictive for $n$ large enough, and in high-probability.
For example, given $n$ observations sampled according to a Bernoulli distribution with success probability $0.5 < \theta^\star \le 1$, the plug-in predictive will predict all future observations to be failures with $\dgp$-probability $(1 - \theta^\star)^n$.
This corresponds to the worst possible predictive distribution, and \textit{any} $\tau < \infty$ would have performed better.
To illustrate this, we fit a conjugate beta-Bernoulli model with a weakly-informative prior to $n$ samples from a Bernoulli distribution and plot the predictive performance over $0.01\leq\tau\leq100$ in \Cref{fig:beta-binomial-tvd}.
Even for $n = 100$, there are predictively optimal values of $\tau$ across some data replicates, and the choice of temperature does influence the predictive performance in small data samples.
Further, this is not a pathology of discrete data, and it is easy to construct similar examples for the continuous case.
\begin{figure}[t!]
    \centering
    \begin{subfigure}{0.69\linewidth}
        \centering
        \includegraphics{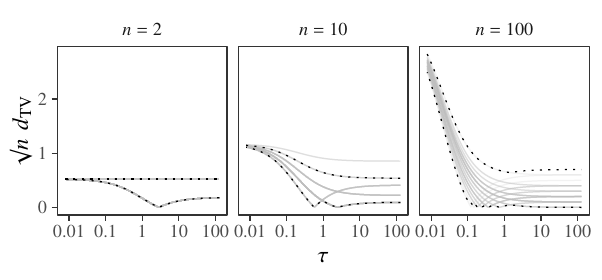}
        \vspace*{-1em}
        \caption{Beta-binomial.}
        \label{fig:beta-binomial-tvd}
    \end{subfigure}
    \hfill
    \begin{subfigure}{0.3\linewidth}
        \centering
        \includegraphics{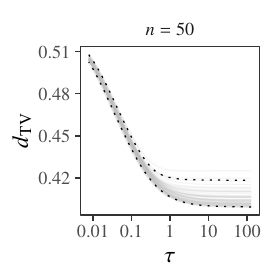}
        \vspace*{-2em}
        \caption{Linear regression.}
        \label{fig:linear-regression-tvd}
    \end{subfigure}
    \caption{Total variation between the  true predictive $\preddens_n(\cdot\mid y_{1:n})$ and the power posterior predictive $p_n^{(\tau)}(\cdot\mid y_{1:n})$ in (a) the beta-binomial experiment and (b) a linear regression experiment.
    Grey curves correspond to individual dataset replicates, and dotted  lines to $5\%$ and $95\%$ quantiles.}
\end{figure}

The previous numerical examples have considered only well-specified models.
Our theory, however, makes no assumptions on correct model specification, and holds under model misspecification.
In \Cref{fig:linear-regression-tvd}, for example, we present the total variation distance for a misspecified linear regression under a weakly-informative prior on the regression coefficients.
Complete numerical results are presented in \Cref{sec:experiment-repetitions} of the supplementary material.
\section{Discussion}\label{sec:discussion}
Our results constitute formal evidence of the common folklore that, in terms of predictive accuracy, parameter uncertainty is of second-order importance relative to data and model uncertainty.
The dominant paradigm in statistical forecasting is to produce probabilistic forecasts that maximise the \textit{sharpness} of the predictions, subject to them being \textit{calibrated} \citep{gneiting_probabilistic_2007}.
Sharpness measures the concentration of the predictive around likely values of the future random variable, and calibration is a frequentist notion of coverage regarding future predictions.
Our study formally demonstrates that, at least in moderate sample sizes, neither sharpness nor calibration of the posterior predictive obtained via the power posterior $\pi^{(\tau)}_n(\theta\mid\y)$ depend on the temperature $\tau$ in any meaningful way.

This is important since, under model misspecification, the oracle model $f_{\theta^\star}$ is not calibrated, and thus nor is the tempered posterior predictive.
As such, likelihood tempering is not sufficient to remedy the lack of calibration due to model misspecification.
This was already noted by \citet{syring_calibrating_2019}, and later investigated by \citet{wu_calibrating_2021} who proposed the twice-tempered predictive $\int f_\theta(\cdot\mid\y)^\tau {\dt\pi_n^{(\tau)}(\theta\mid \y)}$ to resolve this.

Future work will explore which other settings our results extend to.
For instance, the coarsened posteriors introduced by \citet{miller_robust_2019} explicitly prohibit posterior concentration, and thus aren't covered by our assumptions.
It is also not clear for which hierarchical models our results hold, or how our results map to statistically non-identifiable models like Bayesian neural networks.
This is a particularly poignant question since the phenomenon summarised in \Cref{fig:normal-location-tvd} was previously noted empirically in this setting and called the \textit{cold posterior effect} \citep{wenzel_how_2020,aitchison_statistical_2021}.
While the cold posterior effect was thought to be a pathology of Bayesian neural networks, we have rigorously proven it to be a much more universal phenomenon.
Importantly, the assumptions we imposed to do so are never met in Bayesian neural networks, suggesting that the effect may be recoverable under far weaker assumptions.

\subsection*{Acknowledgments}
We are grateful for enlightening discussions with Prof. Pierre Alquier, Dr. Saifuddin Syed, and Dr. Jun Yang which greatly improved this manuscript.
We further thank Prof. Ryan Martin, an anonymous referee, and the Associate Editor for their helpful suggestions.
YM is supported by EP/V521917/1,
JK by EP/W005859/1 and EP/Y011805/1,
and DTF by DE200101070.
\bibliography{main}
\newpage
\begin{appendices}
\counterwithin{figure}{section}
\section{Technical results}
\subsection{Notation}
As mentioned in \Cref{sec:gen-bayes}, power posteriors are a special case of generalised Bayes posteriors where we update our prior $\pi(\theta)$ through an arbitrary loss function $\Ln$, which can be different to the negative log-likelihood.
Our results hold for so-called generalised Bayes posteriors beyond just power posteriors.
Indeed, we can replace any mention of a power posterior $\pi_n^{(\tau)}(\theta\mid\y)\propto\pi(\theta)f_{\theta}(\y)^\tau$ with a generalised Bayesian posterior $\pi_n^{(\tau)}(\theta\mid\y,\Ln)\propto\pi(\theta)\exp\{-\tau\Ln(\theta,\y)\}$ in our assumptions, and proceed by studying the generalised posterior predictive $p_n^{(\tau)}(\cdot\mid\y,\Ln) = \int f_{\theta}(\cdot\mid\y)\dt\pi_n^{(\tau)}(\theta\mid\y,\Ln)$.
In doing so, we show that our results hold for a wide range of posteriors.
\subsection{Main results}
\label{sec:main-results}
In~\Cref{ass:concentration}, we assumed that the risk-minimiser $\hat\theta_n$ converges to the population-optimal $\theta^\star$, and that the Gibbs posterior under that same risk concentrates around a neighbourhood of this $\theta^\star$.
\citet[Theorem 2]{martin_direct_2022} derived this assumption from two sufficient conditions:
that the risk converges uniformly in $\theta$ in $\dgp$-probability, for any $\delta > 0$ and for some small $\beta > 0$,
\begin{equation}\label{ass:uniform-convergence}
    \Esqr{ \sup_{\theta:\,d(\theta,\theta^\star) >\delta} \left\vert n^{-1/2}\{\Ln(\theta) - \Ln(\theta^\star) - \mathsf{L}(\theta) + \mathsf{L}(\theta^\star)\} \right\vert } \lesssim \delta^\beta,
\end{equation}
and, an identifiability condition on $\theta$ such that for a small $\alpha > 0$,
\begin{equation}\label{ass:identifiability}
    \inf_{\theta:\,d(\theta,\theta^\star) >\delta} \{\mathsf{L}(\theta) - \mathsf{L}(\theta^\star)\} \gtrsim \delta^\alpha.
\end{equation}
Under these conditions, \citet[Theorem 5.7, Lemma 19.24]{vaart_asymptotic_1998} shows that $\hat\theta_n$ converges to $\theta^\star$.
And \Cref{ass:uniform-convergence,ass:identifiability} taken together closely resembles~\Cref{ass:concentration_alt1} with $\tau_n$ fixed.
Under an additional weak prior mass assumption, which resembles our~\Cref{ass:concentration_alt2}, \citet[Theorem 2]{martin_direct_2022} show that the Gibbs posterior $\propto\exp\{-\Ln(\theta,\y)\}\pi(\theta)$ concentrates around a neighbourhood of $\theta^\star$.
Their result is similar to our~\Cref{lem:extra_conv} with a fixed $\tau_n > 0$.
In a word, rather than begin with our~\Cref{ass:concentration}, one can start from a set of sufficient conditions which are common to the M-estimation literature and arrive at our result and~\Cref{ass:concentration} that way.

By viewing $\theta$ as risk-minimising parameters more generally than the parameters of a statistical model, we can extend the M-estimation literature \citep[Section 5]{vaart_asymptotic_1998} to include prior information \citep{martin_direct_2022}:
the requirements for a Gibbs posterior under a risk $\Ln$ to concentrate around a neighbourhood of $\theta^\star$ are closely connected to the conditions required for the estimator $\hat\theta_n\to\theta^\star$ under that same risk.
This focus on risk-minimisation is common in generalised Bayesian inference, where certain choices of risk bear richer intuition than the likelihood, something discussed in~\Cref{sec:gen-bayes}.
\begin{proof}[of \Cref{lemma:tvd}]
From the triangle inequality,
\begin{multline}
    \tvdcurl{f_{\hat\theta_n}(\cdot\mid \y)}{p_n^{(\tau)}(\cdot\mid\y,\Ln)} \le \tvdcurl{f_{\theta^{\star}}(\cdot\mid \y)}{f_{\hat\theta_n}(\cdot\mid \y)} \\+ \tvdcurl{f_{\theta^{\star}}(\cdot\mid \y)}{p_n^{(\tau)}(\cdot\mid\y,\Ln)}.\label{eq:triangle-ineq}
\end{multline}
We first show that,  for some constant $C>0$ and with probability at least $1-\exp(-C n\tau\varepsilon_n^2/M_{\varepsilon_n})$, $\tvdcurl{f_{\theta^{\star}}(\cdot\mid \y)}{p_n^{(\tau)}(\cdot\mid\y,\Ln)}\leq \varepsilon_n+2 \exp(-C n\tau\varepsilon_n^2/M_{\varepsilon_n})$.
Recall the relationship between the squared Hellinger distance and squared total variation distance: 
\begin{equation*}
    0\le \tvd{p}{q}^2\le 2\hel{p}{q}^2.
\end{equation*}
Hence, if we can bound $d_{\mathrm{H}}\{f_{\theta^{\star}}(\cdot\mid \y), p_n^{(\tau)}(\cdot\mid\y,\Ln)\}^2$, we have a bound in total variation distance.
From convexity of $q\mapsto \hel{p}{q}^2$ and by Jensen's inequality
\begin{IEEEeqnarray}{rCl}
    \helcurl{f_{\theta^{\star}}(\cdot\mid \y)}{p_n^{(\tau)}(\cdot\mid\y,\Ln)}^2 %
    &=&\helcurl{f_{\theta^{\star}}(\cdot\mid \y)}{\int f_\theta(\cdot \mid \y) \dt \pi_n^{(\tau)}\left(\theta \mid \y, \Ln\right)}^2 \nonumber\\
    &=& \frac{1}{2} \int \bigg[ f_{\theta^{\star}}(x \mid \y)^{1/2} \nonumber\\
    &&\quad\quad\quad- \left\{\int f_\theta(x \mid \y) \dt \pi_n^{(\tau)}\left(\theta \mid \y, \Ln\right)\right\}^{1/2} \bigg]^2 \dt \lambda(x) \nonumber\\
    &\le& \frac{1}{2} \iint \left\{ f_{\theta^{\star}}(x \mid \y)^{1/2} - f_\theta(x \mid \y)^{1/2}  \right\}^2 \nonumber\\
    &&\quad\quad\dt \lambda(x) \dt \pi_n^{(\tau)}\left(\theta \mid \y, \Ln\right) \nonumber\\
    &=& \int_\Theta \helcurl{f_{\theta^{\star}}(\cdot\mid \y)}{f_{\theta}(\cdot\mid \y)}^2 \dt \pi_n^{(\tau)}\left(\theta \mid \y, \Ln\right),\nonumber
\end{IEEEeqnarray}
where $\lambda(x)$ is the Lebesgue measure.
Write $\Theta=A_{\varepsilon_n}\cup A_{\varepsilon_n}^c$, where 
\begin{equation*}
    A_{\varepsilon_n}=\left\{\theta\in\Theta: d(\theta,\theta^\star)\le \varepsilon_n /M^{1/2}_{\varepsilon_n}\right\},
\end{equation*}
which is equivalent to choosing $K=M_{\varepsilon_n}^{-1/2}$ for the set defined in \Cref{ass:concentration}, which is valid since, by \Cref{ass:lipz}, $M_{\varepsilon}>0$ for all $\varepsilon>0$. 
Use the fact that the Hellinger distance is bounded above by unity to obtain
\begin{IEEEeqnarray}{rl}
    \int_\Theta \helcurl{f_{\theta^{\star}}(\cdot\mid \y)}{f_{\theta}(\cdot\mid \y)}^2& \dt \pi_n^{(\tau)}\left(\theta \mid \y, \Ln\right)\nonumber\\
    &=  \int_{A_{\varepsilon_n}} \helcurl{f_{\theta^{\star}}(\cdot\mid \y)}{f_{\theta}(\cdot\mid \y)}^2 \dt \pi_n^{(\tau)}\left(\theta \mid \y, \Ln\right) \nonumber\\
    &\quad + \int_{A_{\varepsilon_n}^c} \helcurl{f_{\theta^{\star}}(\cdot\mid \y)}{f_{\theta}(\cdot\mid \y)}^2 \dt \pi_n^{(\tau)}\left(\theta \mid \y, \Ln\right) \nonumber\\
    &\le \int_{A_{\varepsilon_n}} \helcurl{f_{\theta^{\star}}(\cdot\mid \y)}{f_{\theta}(\cdot\mid \y)}^2 \dt \pi_n^{(\tau)}\left(\theta \mid \y, \Ln\right) \nonumber\\
    &\quad+ \int_{A_{\varepsilon_n}^c}\dt\pi_n^{(\tau)}(\theta\mid \y,\Ln)\label{eq:newhell}\\
    &\le \int_{A_{\varepsilon_n}} \helcurl{f_{\theta^{\star}}(\cdot\mid \y)}{f_{\theta}(\cdot\mid \y)}^2 \dt \pi_n^{(\tau)}\left(\theta \mid \y, \Ln\right) \nonumber\\
    &\quad + \exp(-Cn\tau\varepsilon_n^2/ M_{\varepsilon_n}).\nonumber
\end{IEEEeqnarray}
where the last line holds with probability at least $1-\exp(-Cn\tau\varepsilon_n^2/M_{\varepsilon_n})$ by \Cref{ass:concentration}. 

To control the first term above, we use Assumption \ref{ass:lipz} and in particular \Cref{eq:lipz}:
\begin{align}
\int_{A_{\varepsilon_n}} d_{\mathrm{H}}\{f_{\theta^{\star}}(\cdot\mid \y),&\,f_{\theta}(\cdot\mid \y)\}^2 \dt \pi_n^{(\tau)}\left(\theta \mid \y, \Ln\right) \nonumber \\
&= \int_{A_{\varepsilon_n}} \frac{1}{2} \int \left\{ f_{\theta^{\star}}(x \mid \y)^{1/2} - f_{\theta}(x \mid \y)^{1/2} \right\}^2 \dt \lambda(x) \dt \pi_n^{(\tau)}\left(\theta \mid \y, \Ln\right) \nonumber\\
&\le \int_{A_{\varepsilon_n}} M_{\varepsilon_n} d(\theta,\theta^{\star})^2\dt \pi_n^{(\tau)}\left(\theta \mid \y, \Ln\right) \nonumber.
\end{align}
By \Cref{ass:concentration}, for any $\theta\in A_{\varepsilon_n}$, we have that $d(\theta,\theta^\star)\le \varepsilon_n/M_\varepsilon^{1/2}$, and we can therefore replace the term in the integral with this bound, which yields
\begin{flalign}
   \int_{A_{\varepsilon_n}} M_{\varepsilon_n} d(\theta,\theta^{\star})^2\dt \pi_n^{(\tau)}\left(\theta \mid \y, \Ln\right) &\le \int_{A_{\varepsilon_n}} M_{\varepsilon_n}(\varepsilon_n/M_{\varepsilon_n}^{1/2})^2 \dt \pi_n^{(\tau)}\left(\theta \mid \y, \Ln\right) \nonumber\\
   &\le \varepsilon_n^2 \int_{\Theta}\dt \pi_n^{(\tau)}\left(\theta \mid \y, \Ln\right) \nonumber\\
   &\le \varepsilon_n^2.\label{eq:newhell2}
\end{flalign}
To show that the remaining term in \Cref{eq:triangle-ineq} is negligible, we again use Assumption~\ref{ass:lipz} and $\tvd{p}{q}^2\le 2\hel{p}{q}^2$ to obtain that 
\begin{flalign*}
\tvdcurl{f_{\theta^{\star}}(\cdot\mid \y)}{f_{\hat\theta_n}(\cdot\mid \y)}^2 & \le 2\helcurl{f_{\theta^{\star}}(\cdot\mid \y)}{f_{\hat\theta_n}(\cdot\mid \y)}^2\\
& = \int \left\{ f_{\theta^{\star}}(x\mid \y)^{1/2} - f_{\hat\theta_n}(x\mid \y)^{1/2} \right\}^2 \dt\lambda(x) \\
& \le d(\hat\theta_n,\theta^{\star})^2 M_{\nu_n} \\
& \le\nu_n^2
\end{flalign*}
with probability at least $1-\nu_n$. 
The last step comes from \Cref{ass:concentration}.
So, returning to \Cref{eq:triangle-ineq} and plugging in our bounds, we have, 
\begin{flalign*}
    \tvdcurl{f_{\hat\theta_n}(\cdot\mid \y)}{p_n^{(\tau)}(\cdot\mid\y,\Ln)} &\le \tvdcurl{f_{\theta^{\star}}(\cdot\mid \y)}{f_{\hat\theta_n}(\cdot\mid \y)} \nonumber \\
    &\quad + \tvdcurl{f_{\theta^{\star}}(\cdot\mid \y)}{p_n^{(\tau)}(\cdot\mid\y,\Ln)} \\
    &= \left[\tvdcurl{f_{\theta^{\star}}(\cdot\mid \y)}{f_{\hat\theta_n}(\cdot\mid \y)}^2 \right]^{1/2} \nonumber \\
    &\quad +\left[\tvdcurl{f_{\theta^{\star}}(\cdot\mid \y)}{p_n^{(\tau)}(\cdot\mid\y,\Ln)}^2\right]^{1/2} \\
    &\le \left[2\left\{\exp\left( -\frac{Cn\tau\varepsilon_n^2}{ M_\varepsilon} \right) + {\varepsilon_n^2}\right\}\right]^{1/2} + \left(2\nu_n^2\right)^{1/2} \\
    &\le \surd{2}\max\left[\left\{\exp\left( -\frac{Cn\tau\varepsilon_n^2}{ M_\varepsilon} \right)+ {\varepsilon_n^2}\right\}^{1/2}, \nu_n\right].
\end{flalign*}
Lastly, we simplify the first term in the maximum
\begin{equation*}
    \left\{{\varepsilon_n^2}+\exp\left(-\frac{Cn\tau\varepsilon_n^2}{M_{\varepsilon_n}}\right)\right\}^{1/2}\le \left({\varepsilon_n^2}\right)^{1/2}+\left\{\exp\left(-\frac{Cn\tau\varepsilon_n^2}{M_{\varepsilon_n}}\right)\right\}^{1/2}={\varepsilon_n}+\exp\left(-\frac{Cn\tau\varepsilon_n^2}{M_{\varepsilon_n}}\right), 
\end{equation*}
which, since $\surd{2}<2$, yields the stated result.
\end{proof}
\begin{proof}[of~\Cref{lemma:kl}]
Define the set 
\begin{equation}
    A_{\varepsilon_n} = \left\{\theta\in\Theta: \kldcurl{f_{\theta^\star}(\cdot\mid\y)}{f_{\theta}(\cdot\mid\y)}\le \varepsilon_n^2 \right\}.\label{eq:kl-set}
\end{equation}
Since the density $f_\theta(\cdot\mid\y)$ is non-negative, 
\begin{IEEEeqnarray}{rCl}
 p_n^{(\tau)}(\cdot\mid\y,\Ln)&= &\int_\Theta f_\theta (\cdot\mid\y)\dt\pi_n^{(\tau)}(\theta\mid\y,\Ln) \nonumber  \\
 & \ge  &\int_{A_{\varepsilon_n}} f_\theta(\cdot\mid\y) \dt\pi_n^{(\tau)}(\theta\mid\y,\Ln) \label{eq:inequality-Pintau-pAepsn} \\
 &= &\Pi_n^{(\tau)}(A_{\varepsilon_{n}}\mid\y,\Ln)
 p_{A_{\varepsilon_n}}^{(\tau)}(\cdot\mid\y,\Ln),  \nonumber
\end{IEEEeqnarray}
for $p_{A_{\varepsilon_n}}^{(\tau)}(\cdot\mid\y,\Ln)=\int_\Theta f_\theta(\cdot\mid\y)\dt\pi_{A_{\varepsilon_n}}^{(\tau)}(\theta\mid\y,\Ln)$, where 
\begin{equation*}
\Pi^{(\tau)}_n(A_{\varepsilon_n}\mid\y,\Ln)=\int_{A_{\varepsilon_n}}\dt\pi^{(\tau)}_n(\theta\mid\y,\Ln),
\end{equation*}
and, 
\begin{equation*}
\pi^{(\tau)}_{A_{\varepsilon_n}}(\theta\mid\y,\Ln)=\begin{cases}\frac{\pi^{(\tau)}_n(\theta\mid\y,\Ln)}{\Pi^{(\tau)}_n(A_{\varepsilon_{n}}\mid\y,\Ln)}&\text{ if }\theta\in A_{\varepsilon_{n}}\\0&\text{ otherwise}    
\end{cases}
\end{equation*}
denotes the restriction of the posterior to the set $A_{\varepsilon_n}$.
We then have
\begin{flalign}
    \kldcurl{f_{\theta^\star}(\cdot\mid\y)}{p_n^{(\tau)}(\cdot\mid\y,\Ln)}&= \Ecurl[x\sim f_{\theta^\star}(\cdot\mid\y)]{\log f_{\theta^\star}(x\mid\y) - \log p_n^{(\tau)}(x\mid\y,\Ln)}\nonumber \\
    &= \Ecurl[x\sim f_{\theta^\star}(\cdot\mid\y)]{\log f_{\theta^\star}(x\mid\y)} \nonumber \\
    &\quad - \Esqr[x\sim f_{\theta^\star}(\cdot\mid\y)]{\log\left\{ \int f_\theta(x\mid\y)\dt \pi_n^{(\tau)}(\theta\mid\y,\Ln)\right\}}\nonumber \\
    &\leq \Ecurl[x\sim f_{\theta^\star}(\cdot\mid\y)]{\log f_{\theta^\star}(x\mid\y)} \nonumber \\
    &\quad - \Esqr[x\sim f_{\theta^\star}(\cdot\mid\y)]{\log\left\{ p_{A_{\varepsilon_n}}^{(\tau)}(x\mid\y,\Ln)\right\}}\nonumber \\
    &\quad - \Esqr[x\sim f_{\theta^\star}(\cdot\mid\y)]{\log\left\{ \Pi^{(\tau)}_n(A_{\varepsilon_n}\mid\y,\Ln)\right\}}\nonumber \\
    &= \kldcurl{f_{\theta^\star}(\cdot\mid\y)}{p_{A_{\varepsilon_n}}^{(\tau)}(\cdot\mid\y,\Ln)}\nonumber \\
    &\quad -\log \Pi^{(\tau)}_n(A_{\varepsilon_n}\mid\y,\Ln)\nonumber\\
    &\le \int_{A_{\varepsilon}}\kldcurl{f_{\theta^\star}(\cdot\mid \y)}{f_\theta(\cdot\mid\y)}\dt\pi_{A_{\varepsilon_n}}^{(\tau)}(\theta\mid\y,\Ln) \nonumber \\
    &\quad -\log \Pi^{(\tau)}_n(A_{\varepsilon_n}\mid\y,\Ln)\label{eq:kld}
\end{flalign}
where the first inequality follows from \Cref{eq:inequality-Pintau-pAepsn} and the final inequality follows from the convexity of $\kldcurl{f_{\theta^\star}(\cdot\mid\y)}{\cdot}$ in the second argument. 

On $A_{\varepsilon_n}$, $\kldcurl{f_{\theta^\star}(\cdot\mid\y)}{f_\theta(\cdot\mid\y)}$ is bounded above by $\varepsilon_n^2$ by construction \Cref{eq:kl-set}, so that the first term in \Cref{eq:kld} is bounded by $\varepsilon_n^2$. 
For the second term, consider the Taylor expansion
\begin{equation}
    \log (1 - x) = -x -\frac{1}{2}x^2 - o\left(\frac{1}{2}x^2\right).
\end{equation}
Then, by \Cref{ass:concentration}, with probability at least $1-\exp(-Cn\tau\varepsilon_n^2)$,
\begin{align*}
    \log \Pi^{(\tau)}_n(A_{\varepsilon_n}&\mid\y,\Ln)=\log \{1-\Pi^{(\tau)}_n(A_{\varepsilon_n}^c\mid\y,\Ln)\}\\
    &= -\Pi^{(\tau)}_n(A_{\varepsilon_n}^c\mid\y,\Ln)-\frac{1}{2}\Pi^{(\tau)}_n(A_{\varepsilon_n}^c\mid\y,\Ln)^2 + o\left\{\Pi^{(\tau)}_n(A_{\varepsilon_n}^c\mid\y,\Ln)^2\right\}\\
    &= -\exp(-Cn\tau\varepsilon_n^2)-\frac{1}{2}\exp(-2Cn\tau\varepsilon_n^2)+o\left\{\frac{1}{2}\exp(-2Cn\tau\varepsilon_n^2)\right\}.
\end{align*}
Placing both terms into \Cref{eq:kld} then yields 
\begin{equation*}
\kldcurl{f_{\theta^\star}(\cdot\mid\y)}{p_n^{(\tau)}(\cdot\mid\y,\Ln)}\le \varepsilon_n^2 + \exp(-Cn\tau\varepsilon_n^2) + o(1),
\end{equation*}
and we conclude.
\end{proof}
\subsection{Additional results}
\label{sec:additional-results}
\begin{lemma}\label{lemma:tvd-expectation}
Under \Cref{ass:lipz,ass:concentration}, 
\begin{equation*}
    \Esqr[\y\sim\mathbb{P}]{\tvdcurl{f_{\theta^\star}(\cdot\mid\y)}{p_n^{(\tau)}(\cdot\mid\y,\Ln)}} \le\varepsilon_n+o(1),
\end{equation*}
where $\dgp$ is the true data-generating measure.
\end{lemma}
\begin{proof}
The proof follows very similarly to the first part of Lemma~\ref{lemma:tvd}. 
In particular, the result follows if we can bound $\mathbb{E}_{\y\sim\dgp}[d_{\mathrm{H}}\{{f_{\theta^{\star}}(\cdot\mid \y)},{p_n^{(\tau)}(\cdot\mid\y,\Ln)}\}^2]$. 
This is because, for any densities $q,p$,
\begin{flalign*}
\mathbb{E}_{\y\sim\dgp}\{d_{\mathrm{TV}}(p,q)\}&=\mathbb{E}_{\y\sim\dgp}\left[\left\{d_{\mathrm{TV}}(p,q)^2\right\}^{1/2}\right]\\&\le \mathbb{E}_{\y\sim\dgp}\left[\left\{2d_{\mathrm{H}}(p,q)^2\right\}^{1/2}\right]\\&\le 2\left[\mathbb{E}_{\y\sim\dgp}\{d_{\mathrm{H}}(p,q)^2\} \right]^{1/2}
\end{flalign*}
where the first inequality follows from $d_{\mathrm{TV}}(p,q)^2\le 2d_{\mathrm{H}}(p,q)^2$, and the second from Jensen's inequality. 
We then see that the result will follow if we can show that $\mathbb{E}[ d_{\mathrm{H}}\{{f_{\theta^{\star}}(\cdot\mid \y)},{p_n^{(\tau)}(\cdot\mid\y,\Ln)}\} ^2]\le \varepsilon_n^2+o(1)$.

To this end, recall that, from the proof of \Cref{lemma:tvd}, in particular the arguments used to obtain \Cref{eq:newhell}, for $A_{\varepsilon_n}=\{\theta\in\Theta: d(\theta,\theta^\star)\le \varepsilon_n /M^{1/2}_{\varepsilon_n}\}$,
\begin{IEEEeqnarray}{rl}
    d_{\mathrm{H}}\{{f_{\theta^{\star}}(\cdot\mid \y)},\,&{p_n^{(\tau)}(\cdot\mid\y,\Ln)}\}^2 \nonumber \\
    &\leq \int_\Theta d_{\mathrm{H}}\{{f_{\theta^{\star}}(\cdot\mid \y)},{f_\theta(\cdot\mid\y)}\}^2 \dt \pi_n^{(\tau)}\left(\theta \mid \y, \Ln\right) \nonumber \\
    &\leq \int_{A_{\varepsilon_n}}d_{\mathrm{H}}\{{f_{\theta^{\star}}(\cdot\mid \y)},{f_\theta(\cdot\mid\y)}\}^2\dt \pi_n^{(\tau)}(\theta\mid\y,\mathsf{L}_n) \nonumber\\
    &\quad+ \int_{A_{\varepsilon_n}^c} \dt \pi_n^{(\tau)}(\theta\mid\y,\mathsf{L}_n).\label{eq:expect_ver}
\end{IEEEeqnarray}
From Assumption~\ref{ass:concentration}, we have that
\begin{equation*}
    0\le \int_{A_{\varepsilon_n}^c} \dt \pi_n^{(\tau)}(\theta\mid\y,\Ln)\le \exp(-Cn\tau\varepsilon_n^2/ M_{\varepsilon_n})
\end{equation*}
with probability at least $1-\exp(-Cn\tau\varepsilon_n^2/M_{\varepsilon_n})$. 
Since the right-hand side of the above does not depend on $\y$, we can apply the dominated convergence theorem to obtain
\begin{equation}\label{eq:expecver1}
\mathbb{E}_{\y\sim\dgp}\left\{\int_{A_{\varepsilon_n}} \dt \pi_n^{(\tau)}(\theta\mid\y,\Ln)\right\}=o(1).
\end{equation}
Lastly, from the steps used to obtain \Cref{eq:newhell2} in the proof of \Cref{lemma:tvd}, we know that 
\begin{equation}\label{eq:expecver2}
    \int_{A_{\varepsilon_n}}d_{\mathrm{H}}\{{f_{\theta^{\star}}(\cdot\mid \y)},{f_\theta(\cdot\mid\y)}\}^2\dt \pi_n^{(\tau)}(\theta\mid\y,\Ln)\le \varepsilon_n^2.
\end{equation}
Using \Cref{eq:expecver1} and \Cref{eq:expecver2} and taking expectations of both sides of \Cref{eq:expect_ver} yields the stated result.
\end{proof}

As discussed in \Cref{sec:technical-results} of the main text, a version of \Cref{lemma:tvd} is maintained if we instead consider a learning rate $\tau_n>0$ and allow $\tau_n\rightarrow0$ as $n\rightarrow\infty$. 
In particular, the result we will derive next shows that if we take a learning rate converging to zero, then the posterior concentration rate is reduced in accordance with the rate at which $\tau_n$ converges to zero. 
To show this formally, write $\Ln(\theta)$ to denote an arbitrary loss, suppressing the explicit dependence on $\y$ for simplicity, and rewrite the (generalised) posterior as
\begin{flalign*}
    \pi^{(\tau)}_n(\theta\mid\y,\Ln)=\frac{\pi(\theta)\exp\{n\tau \Ln(\theta)\}}{\int_\Theta \pi(\theta)\exp\{n\tau \Ln(\theta)\}\dt\theta}&=\frac{\pi(\theta)\exp[n\tau \{\Ln(\theta)-\Ln(\theta^\star)\}]}{\int_\Theta \pi(\theta)\exp[n\tau \{\Ln(\theta)-\Ln(\theta^\star)\}]\dt\theta}\\&=\frac{\pi(\theta)\exp[n\tau \{\Ln(\theta)-\Ln(\theta^\star)\}]}{Z_n}.
\end{flalign*}
We prove that a variant of \Cref{ass:concentration} remains satisfied, under the following regularity conditions, if we take a learning sequence $\tau_n>0$ and allow $\tau_n\rightarrow0$. 
The assumptions and lemmata we posit to arrive at this result are broadly similar to those of \citet[Theorem 3]{syring_gibbs_2023}, \citet[Theorem 2]{martin_direct_2022} and \citet[Theorem 2]{shen_rates_2001}, where $n$ is replaced by $n\tau_n$.
\begin{assumption}\label{ass:concentration_alt1}
    There exists a positive sequence $s_n\rightarrow0$ and constant $c_1>0$ such that, for $ns_n^2\tau_n\rightarrow\infty$, 
    \begin{equation*}
        \dgp\left[\sup_{\theta:d(\theta,\theta^\star)\ge s_n}n\tau_n\left\{\Ln(\theta)-\Ln(\theta^\star)\right\}>-c_1ns_n^2\tau_n\right]=o(1).
    \end{equation*}
\end{assumption}
Define $K_n(\theta,\theta^\prime) = \Ln(\theta) - \Ln(\theta^\prime)$ and $K(\theta,\theta')=\lim_n\Ecurl[\y\sim\dgp]{K_n(\theta,\theta')}$. 
Likewise, define $V(\theta,\theta')=\lim_n\text{var}_{\y\sim\,\mathbb{P}}\{n^{1/2}K_n(\theta,\theta') \}$. 
\begin{assumption}\label{ass:concentration_alt2}
    For some $t_n\rightarrow0$ such that $nt_n\tau_n\rightarrow\infty$, define the set 
    \begin{equation}\label{eq:gn_def}
    G_n =\{\theta\in\Theta: \max\{K(\theta^\star,\theta), V(\theta^\star,\theta)\}\le t_n\}.	
    \end{equation}
    Then 
    \begin{equation} \label{eq:Gn_bound}
        \int_{G_n} \dt\pi(\theta) \gtrsim e^{-2n \tau_nt_n}.
    \end{equation}
\end{assumption}
The following lemma bounds $Z_n$ using \Cref{ass:concentration_alt2} and is similar to Lemma 1 of \citet{syring_gibbs_2023}.
We present this result for completeness. 
\begin{lemma}\label{lem:extra1}
    Under Assumption \ref{ass:concentration_alt2}, if $n t_n\tau_n\rightarrow\infty$, then, with $\dgp$-probability converging to $1$, 
    \begin{equation*}
        \mathbb{P}\left\{Z_n\le  \Pi(G_n)e^{-2 nt_n\tau_n}\right\} \le 2\left(n \tau_nt_n\right)^{-1} .
    \end{equation*}
\end{lemma}
\begin{proof}
    Define
    \begin{equation}\label{eq:Qn_def}
        Q_n(\theta^\star,\theta) = \frac{\{\Ln(\theta^\star) - \Ln(\theta)\} -   K(\theta^\star,\theta)}{ V(\theta^\star,\theta)^{1/2}}=\frac{ K_n(\theta^\star,\theta)-K(\theta^\star,\theta)}{ V(\theta^\star,\theta)^{1/2}},
    \end{equation}
Let
\begin{equation}\label{eq:mathcalQn_def}
        \mathcal{Q}_n = \{(\theta, \y): |Q_n(\theta^\star,\theta)| \geq t_n^{1/2} \}
    \end{equation}
    and define the sets 
    \begin{align}
        \mathcal{Q}_n(\theta) &= \{\y \in \Y: (\theta, \y) \in \mathcal{Q}_n\},\,\text{and} \nonumber\\\quad\mathcal{Q}_n(\y) &= \{\theta \in \Theta: (\theta, \y) \in \mathcal{Q}_n\}. \label{eq:Qny_def}
    \end{align}
    Write $Z_n$ as
    \begin{flalign*}
        Z_n = &\int_\Theta \exp\{n\tau_n K_n(\theta,\theta^\star)\}\dt\pi(\theta) \\ 
        =&\int_\Theta \exp\left[\frac{-n\tau_nV(\theta^\star,\theta)^{1/2}\{ K_n(\theta^\star,\theta)-K(\theta^\star,\theta)\}}{V(\theta^\star,\theta)^{1/2}}\right]\exp\{-n\tau_n K(\theta^\star,\theta)\}\dt\pi(\theta) \\
        =&\int_\Theta \exp\{- n\tau_nV(\theta^\star,\theta)^{1/2}Q_n(\theta^\star,\theta)\}\exp\{-n\tau_n K(\theta^\star,\theta)\}\dt\pi(\theta)
    \end{flalign*}
    By \Cref{eq:mathcalQn_def}, on the set $\{\theta\in\Theta:G_n \cap \mathcal{Q}_n(\y)^c\}, -\vert Q_n(\theta^\star, \theta)\vert\ge -t^{1/2}_n$, and since $\vert Q_n(\theta^\star, \theta) \vert \geq Q_n(\theta^\star, \theta)$, $\exp\{-Q_n(\theta^\star,\theta)\}\ge\exp\{-\vert Q_n(\theta^\star,\theta) \vert\}\ge\exp(-t_n^{1/2})$.  Similarly we can bound $V(\theta^\star,\theta) \leq t_n$ and $K(\theta^\star,\theta) \leq t_n$ by \Cref{eq:gn_def}.
    Hence,
    \begin{flalign*}
        Z_n &\geq \int_{G_n \cap \mathcal{Q}_n(\y)^c}\exp\{- \tau_nV(\theta^\star,\theta)^{1/2}Q_n(\theta^\star,\theta)\}\exp\{-n\tau_n K(\theta^\star,\theta)\}\dt\pi(\theta) \\
        &\geq e^{-2\tau_n n t_n} \int_{G_n \cap \mathcal{Q}_n(\y)^c}\dt\pi(\theta)\\
        &=e^{-2\tau_n n t_n}\left[\Pi(G_n)- \Pi\{{G_n \cap \mathcal{Q}_n(\y)}\}\right],
    \end{flalign*}
    where $\Pi(A) = \int_A \dt\pi(\theta)$ and $\pi(\theta)$ is the prior.
    So we now have that
    \begin{flalign*}
        \mathbb{P}\left\{Z_n \le \Pi(G_n) e^{-2n\tau_n t_n}\right\}	&\le \mathbb{P}\left(e^{-2\tau_n n t_n}\left[\Pi(G_n)- \Pi\{{G_n \cap \mathcal{Q}_n(\y)}\}\right]\le \Pi(G_n) e^{-2n\tau_n t_n}\right)\\
        &=\mathbb{P}\left[e^{-2\tau_n n t_n} \Pi\{{G_n \cap \mathcal{Q}_n(\y)}\}\ge \frac{1}{2}\Pi(G_n) e^{-2n\tau_n t_n}\right]\\
        &=\mathbb{P}\left[ \Pi\{{G_n \cap \mathcal{Q}_n(\y)}\}\ge \frac{1}{2}\Pi(G_n)\right].
    \end{flalign*}
    By Markov's inequality,
    \begin{equation}\label{eq:prob_set_markov}
        \mathbb{P}\left[ \Pi\{{G_n \cap \mathcal{Q}_n(\y)}\}\ge \frac{1}{2}\Pi(G_n)\right] \le \frac{2\Esqr[\y\sim\dgp] {\Pi\{{G_n \cap \mathcal{Q}_n(\y)}\}}}{\Pi(G_n)}
    \end{equation}
    and we must therefore control $\Esqr[\y\sim\dgp]{\Pi\{{G_n \cap \mathcal{Q}_n(\y)}\}}$.
   By Fubini's theorem, 
\begin{flalign*}
	\Esqr[\y\sim\dgp]{\Pi\{{G_n \cap \mathcal{Q}_n(\y)}\}}&=\int_{\Y} \int_\Theta \mathds{1}_{\{\theta\in G_n\cap\mathcal{Q}_n(\y)\}}\dt\pi(\theta)\dt\mathbb{P}(\y)\nonumber\\&=\int_{\Y} \int_\Theta \mathds{1}_{\{\theta\in G_n\}} \mathds{1}_{\left\{\theta\in\mathcal{Q}_n(\y)\right\}}\dt\pi(\theta)\dt\mathbb{P}(\y)\nonumber\\ &=\int_{\Y} \int_{G_n}  \mathds{1}_{\left\{\theta\in\mathcal{Q}_n(\y)\right\}}\dt\pi(\theta)\dt\mathbb{P}(\y)\\ &=\int_{G_n}\int_{\Y}   \mathds{1}_{\left\{\theta\in\mathcal{Q}_n(\y)\right\}}\dt\mathbb{P}(\y)\dt\pi(\theta)
\end{flalign*}
If $\theta,\y\notin \mathcal{Q}_n$, then $\mathds{1}_{\left\{\theta\in\mathcal{Q}_n(\y)\right\}}=0$. 
As such, for $\y\in\mathcal{Y}^n$, the integrand only takes non-zero values on the joint event $(\theta,\y)\in\mathcal{Q}_n$.
Therefore, over $(\theta,\y)\in G_n\times \mathcal{Y}^n$, the integrand is non-zero only on the event $(\theta,\y)\in \mathcal{Q}_n$, which yields
\begin{flalign*}
	\Esqr[\y\sim\dgp]{\Pi\{{G_n \cap \mathcal{Q}_n(\y)}\}}&=\int_{G_n}\int_{\Y} \mathds{1}_{\left\{\theta\in\mathcal{Q}_n(\y)\right\}}\dt\mathbb{P}(\y)\dt\pi(\theta)\\&=	\int_{G_n}\mathbb{P}\left\{\y\in\Y:\y\in \mathcal{Q}_n(\theta)\right\} \dt\pi(\theta)
\end{flalign*}
Now, by Markov's inequality, 
\begin{flalign*}
\mathbb{P}\left\{\y\in\Y:\y\in \mathcal{Q}_n(\theta)\right\}&=\mathbb{P}\left\{\y\in\Y: \vert Q_n(\theta^\star,\theta)\vert \ge t_n^{1/2}\right\}\\
&=\mathbb{P}\left\{\y\in\Y: \frac{ \vert K_n(\theta^\star,\theta)-K(\theta^\star,\theta) \vert}{ V(\theta^\star,\theta)^{1/2}}\ge t_n^{1/2}\right\}\\
&=\mathbb{P}\left\{\y\in\Y: \frac{ n^{1/2}\vert K_n(\theta^\star,\theta)-K(\theta^\star,\theta)\vert}{ V(\theta^\star,\theta)^{1/2}}\ge (nt_n)^{1/2}\right\}\\
&=\mathbb{P}\left\{\y\in\Y: \frac{ [n^{1/2}\{K_n(\theta^\star,\theta)-K(\theta^\star,\theta)\}]^2}{ V(\theta^\star,\theta)}\ge nt_n\right\}\\
&\le \frac{1}{nt_n} \frac{1}{V(\theta^\star,\theta)}\E[\y\sim\dgp]{[n^{1/2}\{K_n(\theta^\star,\theta)-K(\theta^\star,\theta)\}]^2}\\
&=1/(nt_n)
\end{flalign*}
where the expectation in the inequality comes from the definitions of $K(\theta^\star,\theta)$ and $V(\theta^\star,\theta)$.
Hence, we can bound
\begin{flalign}
	\Esqr[\y\sim\dgp]{\Pi\{{G_n \cap \mathcal{Q}_n(\y)}\}}&=\int_{G_n}\mathbb{P}\left\{\y:\y\in \mathcal{Q}_n(\theta)\right\} \dt\pi(\theta)\nonumber\\&\le \frac{1}{n t_n} \int_{G_n}\pi(\theta)\dt\theta\nonumber\\&=\frac{\Pi(G_n)}{n t_n}.\label{eq:markov2}
\end{flalign}
Applying \Cref{eq:markov2} to \Cref{eq:prob_set_markov},
\begin{flalign*}
	\mathbb{P}\left\{Z_n \le \Pi(G_n) e^{-2n\tau_n\ t_n}\right\}	&\le \mathbb{P}\left[ \Pi\{{G_n \cap \mathcal{Q}_n(\y)}\}\ge \frac{1}{2}\Pi(G_n)\right]\\
    & \le \frac{2\Esqr[\y\sim\dgp] {\Pi\{{G_n \cap \mathcal{Q}_n(\y)}\}}}{\Pi(G_n)}\\
    & \le \frac{2}{\Pi(G_n)} \frac{\Pi(G_n)}{n t_n}\\&=\frac{2}{n t_n}	
\end{flalign*}	
as stated.
\end{proof}
Define the set $A_\varepsilon=\{\theta\in\Theta:d(\theta,\theta^{\star})\le K\varepsilon\}$ for $K$ arbitrarily large. 
Using \Cref{lem:extra1,ass:concentration_alt1} we achieve a version of \Cref{ass:concentration} in the main text where $\tau_n$ is allowed to shrink to zero. 
This result is similar to Theorem 2 in \citet{shen_rates_2001} and Theorem 3 in \citet{syring_gibbs_2023}.
\begin{lemma}\label{lem:extra_conv}
    Under \Cref{ass:concentration_alt1,ass:concentration_alt2}, for $\varepsilon_n= \max (s_n, t_n^{1/2})$, if $n\tau_n\varepsilon_n^2\rightarrow\infty$, for $K>0$ and sufficiently large, $c>0$, and $n$ sufficiently large, 
    \begin{equation*}
        \int_{A_{\varepsilon_n}^c} \dt\pi^{(\tau_n)}_n(\theta\mid\y,\Ln)\lesssim \exp(-nK^2 c \tau_n\varepsilon_n^2/2)
    \end{equation*}
    with $\dgp$-probability converging to one.
\end{lemma}
\begin{proof}
    For $\varepsilon_n$ as in the statement of the proof, 
    \begin{flalign} \label{eq:restricted_posterior_def}
        \int_{A_{\varepsilon_n}^c} \dt\pi^{(\tau_n)}_n(\theta\mid\y,\Ln) = & \frac{1}{Z_n}\int_{A_{\varepsilon_n}^c}\exp[n\tau_n\{\Ln(\theta)-\Ln(\theta^\star)\}]\dt\pi(\theta).
    \end{flalign}    
    By \Cref{lem:extra1},
    \begin{equation}\label{eq:bound_from_ls2}
        \mathbb{P}\left\{Z_n \le \Pi(G_n) e^{-2n\tau_n t_n}\right\}\le \frac{2}{nt_n}.
    \end{equation}
    Now, by \Cref{ass:concentration_alt1}, with probability converging to one,
    \begin{flalign}\label{eq:bound_from_as1}
        \int_{A_{\varepsilon_n}^c}\exp[n\tau_n\{\Ln(\theta)-\Ln(\theta^\star)\}]\dt\pi(\theta) \le \exp\{-cn\tau_n (K \varepsilon_n)^2\},
    \end{flalign}
    since $K$ is large and $\varepsilon_n\ge s_n$.
    Hence, applying \Cref{eq:bound_from_ls2} and \Cref{eq:bound_from_as1} to \Cref{eq:restricted_posterior_def}, with $\dgp$-probability tending to one,
    \begin{equation*}
        \int_{A_{\varepsilon_n}^c} \dt\pi^{(\tau_n)}_n(\theta\mid\y,\Ln)\le \frac{e^{-cn\tau_nK^2\varepsilon_n^2}}{\Pi(G_n) e^{-2n\tau_n t_n}}.
    \end{equation*}
    Finally, by \Cref{ass:concentration_alt2}, $\Pi(G_n) \gtrsim \exp(-2n\tau_n t_n)$, so that $1 / \Pi(G_n) \lesssim 1 / \exp(-2n\tau_n t_n)$, and thus
    \begin{equation*}
        \int_{A_{\varepsilon_n}^c} \dt\pi^{(\tau_n)}_n(\theta\mid\y,\Ln) \le  \frac{e^{-cn\tau_nK^2\varepsilon_n^2}}{ e^{-4n\tau_n t_n}}.
    \end{equation*}
    For $K>0$ large enough so that $t_n\le \varepsilon_n^2\le cK^2\varepsilon_n^2/8$, we have that
    \begin{align*}
        \int_{A_{\varepsilon_n}^c} \dt \pi^{(\tau_n)}_n(\theta\mid\y,\Ln) &\le \frac{e^{-cn\tau_nK^2\varepsilon_n^2}}{ e^{-4n\tau_n t_n}} \\
        &\le \frac{e^{-cn\tau_nK^2\varepsilon_n^2}}{ e^{-4n\tau_n \varepsilon_n^2}} \\
        &\le \frac{e^{-cn\tau_nK^2\varepsilon_n^2}}{ e^{-4n\tau_n cK^2\varepsilon_n^2/8}} \\
        &= e^{-cn\tau_nK^2\varepsilon_n^2/2}
    \end{align*}
    with $\dgp$-probability tending to one. 
\end{proof}
\Cref{lem:extra1} shows that if $\tau_n\rightarrow0$ as $n\rightarrow\infty$, then the posterior concentration rate is determined by the condition $n\tau_n\varepsilon_n^2\rightarrow\infty$.
The rate of posterior concentration can then be expressed as $\varepsilon_{\tau,n}=(n\tau_n)^{-1/2}\log(n)$. 
Hence, if we take $\tau_n=n^{-1/2}$, the rate of posterior concentration, as determined by $\varepsilon_{\tau,n}$, cannot exceed $n^{-1/4}$. 
This illustrates that choosing a learning rate sequence  $\tau_n\rightarrow0$ essentially reduces your sample size from $n$ to $n\tau_n$, and reduces the rate of posterior concentration accordingly. 

A consequence of \Cref{lem:extra1} is that the results on the accuracy of the posterior predictive given in \Cref{lemma:tvd,lemma:kl} in the main text remain valid when we replace $\varepsilon_n$ in those results with $\varepsilon_{\tau,n}$. 
Hence, even if we allow $\tau_n\rightarrow0$, a version of the stated results in the main text will remain valid so long as $\tau_n$ does not go to zero faster than $\log(n)/n$.
\section{Experiment details and repetition with cross-validation}\label{sec:experiment-repetitions}
In the following section, we drop the dependence on an arbitrary loss function $\Ln$ and treat only the power posterior case as in the main text.
\subsection{Normal location example}\label{sec:normal-location-supp}
In order to demonstrate the flatness proven above and referred to in \Cref{sec:normal-location-example} of the main text, we empirically show the distance of the posterior predictive distribution, $p_n^{(\tau)}(\cdot\mid\y)$, to the true predictive, $\preddens_n(\cdot\mid\y)$.
We do this for a range of $\tau$, two different prior distributions, and two distances: the total variation distance, and the Kullback-Leibler divergence.

We simulate data from a $\normal(\theta^{\star} =0,{\sigma^\star}^2 = 1)$ distribution and consider the likelihood $\normal(\y ;\, \theta,\sigma = 1)$, so that the model is well-specified.
We define a weakly-informative prior $\pi(\theta) = \normal(\theta ;\, 0,\sigma_0^2)$. 
Letting $\sigma_0\to\infty$ results in a flat prior over the parameter space.
In this case, the posterior predictive density is $\normal(\cdot ;\, \mu_{n}, \sigma^2 + \sigma_{n}^2)$ where 
\begin{equation*}
    \sigma_{n}^2 = \frac{1}{n\tau / \sigma^2 + 1 / \sigma_0^2},\quad
    \mu_{n} = \sigma_{n}^2 \left( \frac{\mu_0}{\sigma_0^2} + \frac{n\tau\bar{y}_{1:n}}{\sigma^2} \right) .
\end{equation*}
Here and throughout, $\bar{y}_{1:n}$ denotes the sample mean over $n$ observations. 
We simulate $1,000$ data replicates from the true distribution and compute the total variation distance by quadrature integration for the experiments in the main text.

Now for the Kullback-Leibler divergence, we have that
\begin{align*}
    \kldcurl{\preddens_n(\cdot\mid\y)}{p_n^{(\tau)}(\cdot \mid y_{1:n})} &= \frac{1}{2}\log (1+\sigma^2_{n}) + \frac{1 + \mu_{n}^2}{2(1+\sigma_{n}^2)}- \frac{1}{2}.
\end{align*}
We can analytically compute the expectation under $y_{1:n} \sim \dgp$ to achieve the so-called \textit{risk} as a function of $\tau$,  
\begin{align*}
    \riskBiomet(\tau)&= \Esqr[y_{1:n}\sim \dgp]{\kldcurl{\preddens_n(\cdot\mid\y)}{p_n^{(\tau)}(\cdot \mid y_{1:n})}} \\
    &=\frac{1}{2}\log (1+\sigma^2_{n}) + \frac{1 + n\tau\sigma_{n}^2 \Ecurl[y_{1:n}\sim \dgp]{(\bar{y}_{1:n})^2}}{2(1+\sigma_{n}^2)}- \frac{1}{2}\\
    & =\frac{1}{2}\log (1+\sigma^2_{n}) + \frac{1 + \tau\sigma_{n}^2 }{2(1+\sigma_{n}^2)}- \frac{1}{2}\cdot
\end{align*}
If we take $\sigma_0^2 \to \infty$ to simplify this (which is equivalent to the prior disappearing with $n$), we get
\begin{equation}\label{eq:normal-location-risk}
    \riskBiomet(\tau) =\frac{1}{2}\log \left( 1+\frac{1}{n\tau} \right) + \frac{1 + n^{-1} }{2\{1+(n\tau)^{-1}\}}- \frac{1}{2}.
\end{equation}
The derivative of which with respect to $\tau$ is
\begin{equation*}
    \frac{\partial}{\partial \tau}\riskBiomet(\tau) = \frac{1}{2n}\left\{\frac{\tau(1+n^{-1}) - (\tau+n^{-1})}{\tau(\tau+n^{-1})^2}\right\} = \frac{1}{2n}\frac{(\tau -1)n^{-1}}{\tau(\tau+n^{-1})^2},
\end{equation*}
which is negative for $\tau < 1$, 0 at $\tau = 1$ and increasing for $\tau > 1$. 
This means a theoretical optimum is attained at $\tau = 1$.

\begin{figure}[t!]
    \centering
    \includegraphics{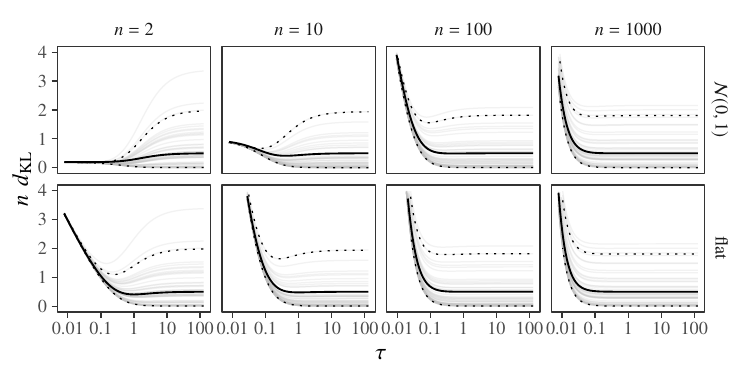}
    \caption{Normal location example. The grey curves correspond to individual dataset replicates,  dotted black lines to $5\%$ and $95\%$ quantiles, and solid black curves to expectation.}
    \label{fig:normal-location-analytic-risk}
\end{figure}
In Figure~\ref{fig:normal-location-analytic-risk} we plot the Kullback-Leibler divergence and the risk across two priors: a weakly-informative $\normal(0, \sigma_0^2 = 1)$ prior, and a flat prior (letting $\sigma_0^2\to\infty$).

As previously mentioned, in practice, we can estimate the Kullback-Leibler divergence with the expected negative log-likelihood, and approximate the inner expectation with cross-validation.
We call this approximation the leave-one-out cross-validation expected log-predictive density \citep{vehtari_practical_2017}, or just the cross-validation score in short.
In this case, the cross-validation score is
\begin{equation}
    \elpdBiomet(\tau) = \frac{1}{n}  \sum_{i=1}^n \log p_n^{(\tau)}(y_i \mid y_{-i}) =  \frac{1}{n} \sum_{i=1}^n \log \normal(y_{i};\, \mu_{-i}, \sigma^2 + \sigma_{-i}^2), \label{eq:normal-elpd}
\end{equation}
where now 
\begin{equation*}
    \sigma_{-i}^2 = \frac{1}{(n - 1)\tau / \sigma^2 + 1 / \sigma_0^2},\quad
    \mu_{-i} = \sigma_{-i}^2\left\{ \frac{\mu_0}{\sigma_0^2} + \frac{(n - 1)\tau\bar{y}_{-i}}{\sigma^2} \right\}.
\end{equation*}

It is worth emphasising that our results are meant only to explain the impact of the temperature parameter on predictive inferences when $n$ is large enough so that posterior concentration is in evidence, and they are not necessarily indicative of the behaviour of the posterior predictive in small samples. 
Indeed, as suggested by the experiments in \cite[Figure 3]{grunwald_inconsistency_2017}, the role of the temperature parameter can have meaningful impacts on predictive performance in small-to-moderate samples, but its impact abates as the sample size increases. 
It is relatively straightforward to illustrate this behaviour in our simple Gaussian example.

\begin{figure}
    \centering
    \includegraphics{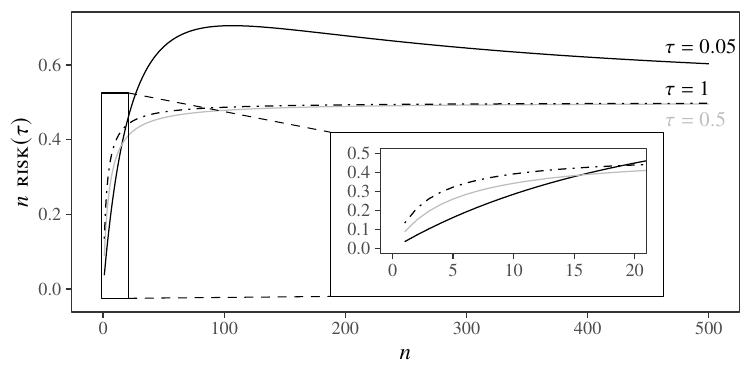}
    \caption{Normal location example. Lines correspond to the scaled risk of~\Cref{eq:normal-location-risk} across different values of $\tau$.}
    \label{fig:normal-location-pre-asymp}
\end{figure}

Consider a prior centred on the true parameter value $\pi(\theta) = \normal(\theta; \theta^\star, \surd 0.25^2)$.
In~\Cref{fig:normal-location-pre-asymp} we see how, once $n$ is large enough, the risk induced by different values of $\tau$ converge.
While $n$ is small, however, tempering can offer a significant predictive advantage.
In this case, the prior is very good and so assigning it more weight than the data results in better predictive distributions when the data are small.
This is seen in the figure inset, where small values of $\tau$ significantly out-perform the Bayes predictive ($\tau = 1$). 
What our theoretical results (e.g. \Cref{lemma:kl}) show is that, at a rate which depends on the posterior concentration, this predictive difference vanishes in $n$.
And eventually, all choices of $\tau$ result in indistinguishable predictive distributions, although for small $\tau$ this phenomenon only occurs at larger $n$.
As such, the predictive benefits of power posteriors are in finite-samples, and vanish with $n$.
Crucially, this behaviour is independent of whether or not the model is well-specified.
\subsection{Misspecified normal location example}\label{sec:misspecified-normal-location-supp}
\begin{figure}[t!]
    \centering
    \includegraphics{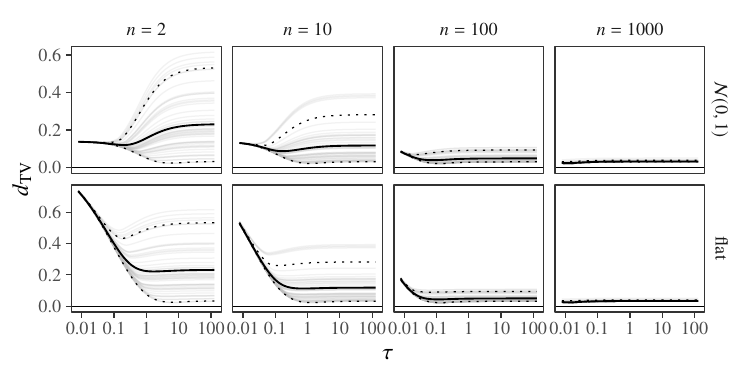}
    \caption{Misspecified normal location example. 
    The grey curves correspond to individual dataset replicates,  dotted black lines to $5\%$ and $95\%$ quantiles, and solid black curves to expectation.}
    \label{fig:normal-location-tvd-misspecified}
\end{figure}
In particular, suppose we instead sample $\y\sim\mathrm{Student-}t_{\nu}$ with $\nu = 10$ degrees of freedom, but keep the same likelihood and priors as before.
This way, the oracle model $f_{\theta^\star}$ is misspecified with tails much lighter than the true data-generating mechanism.
Repeating the experiment in total variation as before results in \Cref{fig:normal-location-tvd-misspecified}.
We find here that, while for small values of $n$ there may be a theoretically optimal value of $\tau$, as $n$ grows the effect of tempering vanishes.
As such, the predictive distribution induced by any particular value of 
$\tau$ quickly resembles $f_{\theta^\star}$. 
And since $\tvdcurl{\preddens_n}{f_{\theta^\star}}>0$, the total variance distance $\tvdcurl{p_n^{(\tau)}}{\preddens_n}$ never reaches zero.
\subsection{Coarsened posteriors}\label{sec:coarsened-posteriors}
\begin{figure}
    \centering
    \includegraphics{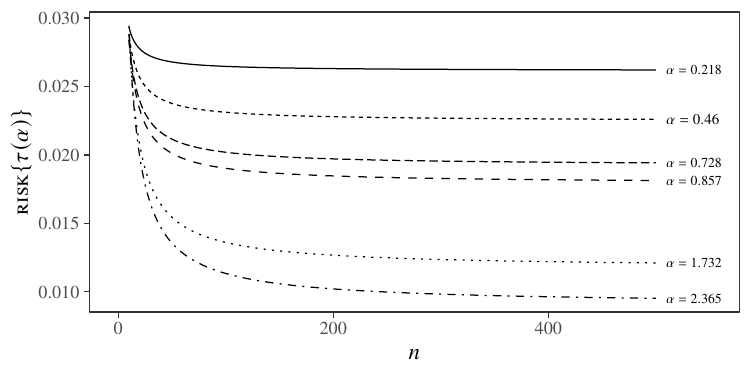}
    \caption{Normal location example. Lines represent the risk of~\Cref{eq:normal-location-risk} under different temperatures of the form $\tau(\alpha) = \alpha/(\alpha + n)$.}
    \label{fig:normal-location-tvd-coarsening}
\end{figure}
The coarsened posteriors of \citet{miller_robust_2019} employ a temperature that decreases to zero at rate $1/n$.
In particular, they take the temperature to be of the form $\tau(\alpha) = \alpha/(\alpha+n)$ where $\alpha>0$ is a hyper-parameter related to the radius of the coarsened set.
We presently recompute the risk in~\Cref{eq:normal-location-risk} for six different samples of $\alpha\sim\mathrm{Exp}(1)$ and $\tau(\alpha)$ as above, and plot the resulting $\riskBiomet\{\tau(\alpha)\}$ in~\Cref{fig:normal-location-tvd-coarsening}.
Since the posterior never concentrates, the limiting risk as $n\to\infty$ is a function of $\alpha$: 
the posterior predictive distributions indexed by different $\alpha$ remain distinct even asymptotically, and our theory does not apply to this case.
\subsection{Linear regression}
Let $\epsilon$ denote the `outlier' rate (which we set to $\epsilon = 0.5$) and $\delta$ the standard deviation of the outlier distribution ($\delta = \surd0.01$) in a linear regression experiment.
We then simulate the predictors by $X_i \sim\normal(0_p, I_p)$, and the target according to
\begin{equation*}
    y_i \sim (1 - \epsilon)\normal(X_i^\top \beta^\star, {\sigma^\star}^2)  + \epsilon \normal(0, \delta^2).
\end{equation*}
We suppress the dependence on the fixed and known constants $\epsilon, \delta$, and $\sigma^\star$ going forward (fixing $\sigma^\star = 1$).
The true parameters are $\beta^\star = (0.1, 0.1, 0.1, 0.1, 0) \in \mathbb{R}^p$, where $p = 5$ and only the first four are relevant.

Consider the regression coefficients $\beta\in\mathbb{R}^p$ with weakly-informative prior $\normal(0_p, \Sigma_0)$.
The likelihood is $f_\beta(y_i\mid X_i) = \normal(y_i;\,\beta^\top X_i, {\sigma^\star}^2)$, and is thus misspecified.
Now for data $\data = \{X, y\}$ with $y\in\mathbb{R}^n$ and $X\in\mathbb{R}^{n\times p}$,
\begin{align*}
    \log \pi_n^{(\tau)}(\beta\mid \data) &= \log \pi(\beta) + \tau \log f_\beta(y\mid X) + C \\
    &= -\frac{1}{2} \beta^\top \Sigma_0^{-1} \beta - \frac{\tau}{2{\sigma^\star}^2}\lVert X\beta - y \rVert^2 + C \\
    &= -\frac{1}{2} \beta^\top \Sigma_0^{-1} \beta - \frac{\tau}{2{\sigma^\star}^2} \left( \beta^\top X^\top X\beta - 2 \beta^\top X^\top y + y^\top y\right) + C \\
    &= -\frac{1}{2}\bigg\{ \beta^\top\underbrace{\left( \Sigma_0^{-1} + \tau{\sigma^\star}^{-2} X^\top X \right)}_{M} \beta - 2 \beta^\top \underbrace{\left( \tau{\sigma^\star}^{-2} X^\top y \right)}_{b} \bigg\} \\
    &= -\frac{1}{2}\left( \beta^\top M \beta - 2\beta^\top b \right) + C \\
    &= -\frac{1}{2}\left \{ (\beta - M^{-1}b)^\top M (\beta - M^{-1}b) \right \} + C \\
    &= -\frac{1}{2} (\beta - \beta_n)^\top \Sigma_n^{-1}(\beta - \beta_n) + C
\end{align*}
for some constant $C$ which does not depend on $\beta$, and where
\begin{align*}
    \beta_n = \tau{\sigma^\star}^{-2}\Sigma_n X^\top y,\quad \Sigma_n^{-1} = \tau{\sigma^\star}^{-2}X^\top X + \Sigma_0^{-1},
\end{align*}
so that the posterior is $\normal(\beta;\,\beta_n, \Sigma_n)$.
Considering the posterior predictive evaluated at new datum $(\tilde{X}, \tilde{y})$
\begin{align}
    p_n^{(\tau)}(\tilde{y}\mid\tilde{X}, \data) &= \int f_\beta(\tilde{y} \mid \tilde{X}) \pi^{(\tau)}_n(\beta\mid\data)\,\dt\beta \nonumber \\
    &= \normal(\tilde{y};\,\beta_n^\top\tilde{X}, \tilde{X}^\top\Sigma_n\tilde{X} + {\sigma^\star}^2). \label{eq:linreg-posterior-pred}
\end{align}
In the leave-one-out case, we have
\begin{equation*}
    \elpdBiomet(\tau) = \frac{1}{n} \sum_{i=1}^n \log\normal(y_i;\;\beta_{-i}^\top X_i, X_i^\top \Sigma_{-i} X_i + {\sigma^\star}^2),
\end{equation*}
where 
\begin{align*}
    \beta_{-i} = \tau{\sigma^\star}^{-2}\Sigma_{-i} X_{-i}^\top y_{-i},\quad \Sigma_{-i}^{-1} = \tau{\sigma^\star}^{-2}X_{-i}^\top X_{-i} + \Sigma_0^{-1}.
\end{align*}
\begin{figure}[t!]
    \centering
    \includegraphics{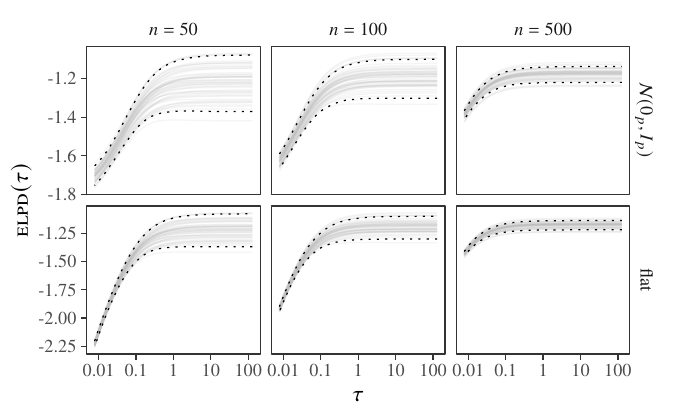}
    \caption{Misspecified linear regression example. 
    The grey curves correspond to individual dataset replicates and dotted black lines to $5\%$ and $95\%$ quantiles.}
    \label{fig:misspecified-linear-regression-elpd}
\end{figure}
This is shown in Figure~\ref{fig:misspecified-linear-regression-elpd} for varying $\tau, n$, and prior choice.

In the above, the $\elpdBiomet(\tau)$ is approximating 
\begin{equation*}
    \int \kldcurl{\preddens_n(\cdot\mid\tilde{X},\data)}{p_n^{(\tau)}(\cdot \mid \tilde{X}, \data)} \dt\dgp(\tilde{X}),
\end{equation*}
while in the total variation case we instead want to analyse
\begin{align*}
    \int \tvdcurl{\preddens_n(\cdot \mid \tilde{X},\data)}{p_n^{(\tau)}\left(\cdot \mid \tilde{X},\data\right)} \dt\dgp(\tilde{X}).
\end{align*}
As such, we can compute the expected total variation distance by using quadrature for the inner expectation (for total variation distance), and Monte Carlo integration for the outer expectation:
\begin{align*}
    \int \tvdcurl{\preddens_n(\cdot \mid \tilde{X},\data)}{p_n^{(\tau)}\left(\cdot \mid \tilde{X},\data\right)} &\dt\dgp(\tilde{X}) \nonumber\\
    &= \int \frac{1}{2}\int\left\vert\preddens_n(\tilde y \mid \tilde{X},\data) - p_n^{(\tau)}\left(\tilde y \mid \tilde{X},\data\right)\right\vert\dt\tilde y \dt\dgp(\tilde{X}) \\
    &\approx \frac{1}{\mathcal{S}} \sum_{s = 1}^\mathcal{S} \left\{\frac{1}{2}\int\left\vert\preddens_n(\tilde y \mid \tilde{X}^{(s)},\data) - p_n^{(\tau)}(\tilde y \mid \tilde{X}^{(s)},\data)\right\vert\dt\tilde y \right\},
\end{align*}
with $\{\tilde{X} ^{(s)}\}_{s = 1}^\mathcal{S}\sim\dgp$ simulated as previously described, and $p_n^{(\tau)}(\tilde y \mid \tilde{X}^{(s)},\data)$ as defined in \Cref{eq:linreg-posterior-pred}, and $\mathcal{S} = 10,000$.
\begin{figure}[t!]
    \centering
    \includegraphics{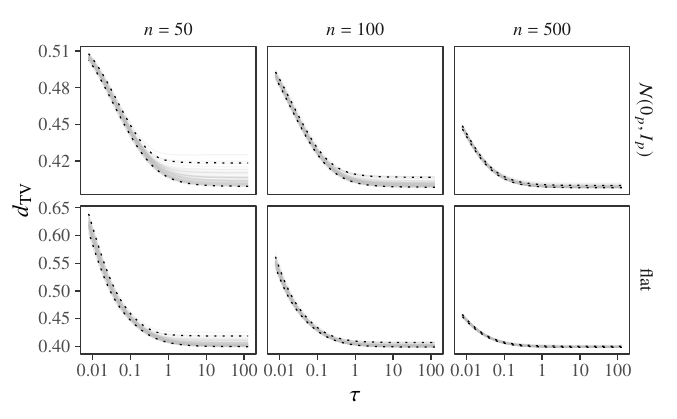}
    \vspace*{-1em}
    \caption{Misspecified linear regression example.
    The grey curves correspond to individual dataset replicates and dotted black lines to $5\%$ and $95\%$ quantiles.}
    \label{fig:misspecified-linear-regression-tvd}
\end{figure}
\subsection{Beta-binomial example}
Consider data sampled according to $y_i\sim \Bernoulli(\theta^\star)$.
Suppose of these $n$ observations we observe $x$ successes and $z = n - x$ failures, then we model the data as coming from a beta-binomial model with prior $\betadist(\alpha, \beta)$.\footnote{We use the shape-scale parameterisation of the beta distribution throughout.}
The posterior distribution under Bayesian inference with the likelihood scaled by $\tau$ following the $n$ observations is then $\pi_n^{(\tau)}(\theta \mid y_{1:n}) \propto \betadist(\theta ;\, \tau x + \alpha, \tau z + \beta)$.
In turn we have the posterior predictive
\begin{IEEEeqnarray*}{rl}
    p_n^{(\tau)}(\ypred \mid y_{1:n}) \,&= \int f_\theta(\ypred)\dt\pi_n^{(\tau)}(\theta \mid y_{1:n})\\
    &= \text{beta-binomial}(1, \tau x + \alpha, \tau z + \beta)
\end{IEEEeqnarray*}
evaluated on the hitherto unseen observation $\tilde{y}$. 
Considering the leave-one-out case, denoting $x_{-i}$ the number of successes with the $i$-th datum deleted, and likewise for $z_{-i}$ the number of failures.
Then the cross-validation score is
\begin{IEEEeqnarray*}{rl}
    \elpdBiomet(\tau) \,&= \frac{1}{n}\sum_{i=1}^n \log p_n^{(\tau)}(y_i \mid y_{-i}) \\
    &=  \frac{1}{n}\sum_{i=1}^n \log \{\text{beta-binomial}(y_i;\,1, \tau x_{-i} + \alpha, \tau z_{-i} + \beta)\}.
\end{IEEEeqnarray*}
\begin{figure}[t!]
    \centering
    \includegraphics{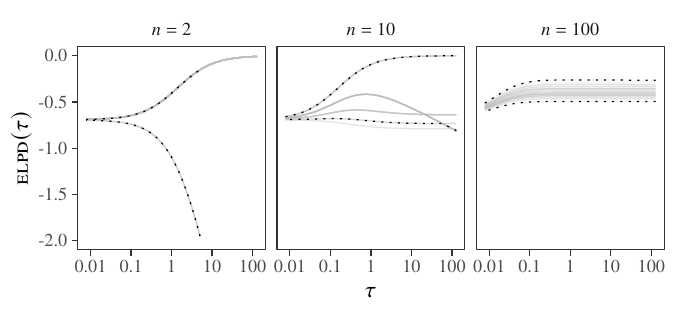}
    \caption{Beta-binomial example. 
    The grey curves correspond to individual dataset replicates and dotted black lines to $5\%$ and $95\%$ quantiles.}
    \label{fig:beta-binomial-elpd}
\end{figure}
In Figure~\ref{fig:beta-binomial-elpd} we show this cross-validation score as a function of $\tau$ across three data regimes and with a $\betadist(1, 1)$ prior.

\end{appendices}
\end{document}